\documentclass[12pt,reqno]{amsart}
\usepackage[margin=1in]{geometry}
\usepackage{amsmath,amssymb,amsthm,graphicx,amsxtra, setspace}
\usepackage[utf8]{inputenc}
\usepackage{mathrsfs}
\usepackage{hyperref}
\usepackage{upgreek}
\usepackage{mathtools}
\usepackage[mathcal]{euscript}
\usepackage{dsfont}
\usepackage{graphicx,type1cm,eso-pic,color}
\allowdisplaybreaks

\usepackage[pagewise]{lineno}

\newtheorem{theorem}{Theorem}[section]

\newtheorem{definition}[theorem]{Definition}

\newtheorem{remark}[theorem]{Remark}

\let\originalleft\left
\let\originalright\right
\renewcommand{\left}{\mathopen{}\mathclose\bgroup\originalleft}
\renewcommand{\right}{\aftergroup\egroup\originalright}

\renewcommand{\d}{\/\mathrm{d}\/}

\def\w{\textbf{W}^{\varepsilon}_{{\theta}^{\varepsilon}}}

\def\S{\mathrm{S}}

\def\L{\mathbb{L}}
\def\A{\mathrm{A}}

\def\F{\mathrm{F}}
\def\C{\mathrm{C}}
\def\f{\mathbf{f}}
\def\J{\mathbb{J}}
\def\B{\mathrm{B}}
\def\D{\mathrm{D}}
\def\y{\mathbf{y}}

\def\Z{\mathrm{Z}}
\def\E{\mathbb{E}}

\def\x{\mathbf{x}}

\def\g{\mathbf{g}}

\def\z{\mathbf{z}}
\def\v{\mathbf{v}}

\def\w{\mathbf{w}}
\def\W{\mathrm{W}}
\def\G{\mathrm{G}}

\def\wi{\widetilde}

\def\u{\mathrm{u}}

\def\u{\mathbf{u}}
\def\H{\mathbb{H}}

\newcommand{\R}{\mathbb{R}}

\renewcommand{\d}{\/\mathrm{d}\/}

\newcommand{\Addresses}{{
		\footnote{
			
			\noindent \textsuperscript{1}Department of Mathematics, Indian Institute of Technology Roorkee-IIT Roorkee,
			Haridwar Highway, Roorkee, Uttarakhand 247667, INDIA.\par\nopagebreak
			\noindent  \textit{e-mail:} \texttt{maniltmohan@ma.iitr.ac.in, maniltmohan@gmail.com.}
			
			\noindent \textsuperscript{*}Corresponding author.

			\textit{Key words:} convective Brinkman-Forchheimer equations, L\'evy noise, fraction Brownian motion, mild solution. 
			
			Mathematics Subject Classification (2010): Primary 76D06; Secondary: 35Q30, 76D03, 47D03.

}}}

\begin{document}
	
	
	\title[Stochastic convective Brinkman-Forchheimer equations]{$\mathbb{L}^p$-solutions of deterministic and stochastic  convective Brinkman-Forchheimer equations
			\Addresses}
	\author[M. T. Mohan ]{Manil T. Mohan\textsuperscript{1*}}

	\maketitle
	
	\begin{abstract}
In the first part of this work, we establish the existence and uniqueness of a local mild solution to the deterministic convective Brinkman-Forchheimer (CBF) equations defined on the whole space, by using properties of the heat semigroup and fixed point arguments based on an iterative technique. The second part is devoted for establishing the existence and uniqueness of a pathwise mild solution upto a random time to the stochastic CBF equations perturbed by L\'evy noise by exploiting the contraction mapping principle. We also discuss the local solvability of the stochastic CBF equations subjected to fractional Brownian noise. 
	\end{abstract}

	\section{Introduction}\label{sec1}\setcounter{equation}{0}
\subsection{Deterministic convective Brinkman-Forchheimer equations}	The Cauchy problem for the convective Brinkman-Forchheimer equations (CBF) in $\mathbb{R}^d,d\geq 2$ can be written as 
	\begin{align}\label{1}
\frac{\partial \u(t,x)}{\partial t}-\mu \Delta\u(t,x)&+(\u(t,x)\cdot\nabla)\u(t,x)+\alpha\u(t,x)+\beta|\u(t,x)|^{r-1}\u(t,x)\nonumber\\+\nabla p(t,x)&=\mathbf{f}(t,x), \ \text{ in } \ (0,T)\times\R^d, 
	\end{align}
	with the conditions
	\begin{equation}\label{2}
	\left\{
	\begin{aligned}
	\nabla\cdot\u(t,x)&=0, \ \text{ in } \ (0,T)\times\R^d, \\
	\u(0,x)&=\u^0(x) \ \text{ in } \ \{0\}\times\R^d,\\
	|\u(t,x)|&\to 0\ \text{ as }\ |x|\to\infty, \ t\in(0,T). 
		\end{aligned}\right.
	\end{equation}
	 In \eqref{1}, $\u(t , x) \in \R^d$ stands for the velocity field at time $t$ and position $x$, $p(t,x)\in\R$ represents the pressure field, $\f(t,x)\in\R^d$ is an external forcing.  The constant $\mu$ denotes the positive Brinkman coefficient (effective viscosity), the positive constants $\alpha$ and $\beta$ represent the Darcy (permeability of porous medium) and Forchheimer (proportional to the porosity of the material) coefficients, respectively. For $\alpha=\beta=0$, we obtain the classical Navier-Stokes equations (NSE). The absorption exponent $r\in[1,\infty)$ and the case $r=3$ is known as the critical exponent.  The critical homogeneous CBF equations \eqref{1} have the same scaling as NSE only when $\alpha=0$ (see Proposition 1.1, \cite{KWH} and no scale invariance property for other values of $\alpha$ and $r$).  Since $\alpha$ does not play a major role in our analysis, we fix $\alpha=0$ and we scale $\mu$ and $\beta$ to unity in the rest of the paper. The existence and uniqueness of weak as well as strong solutions of the system \eqref{1}-\eqref{2} in the whole space and periodic domains is discussed in the works \cite{ZCQJ,PAM,KWH,MTM5}, etc and the references therein.

	 \subsection{Abstract formulation and mild solution}
	 The \emph{Helmholtz-Hodge projection} denoted by $\mathscr{P}$ is a bounded linear operator from  $\L^p(\R^d)$ to $\J_p:=\mathscr{P}\L^p(\R^d)$, $1<p<\infty$. Note that the space $\J_p$ is a separable Banach space with $\L^p(\R^d)$-norm denoted by $\|\cdot\|_p$ and the operator $\mathscr{P}$ is an orthogonal projection of $\L^2(\R^d)$ onto the subspace $\H:=\J_2$. Remember that $\mathscr{P}$ can be expressed in terms of the Riesz transform (cf. \cite{MTSS} for more details). We use the notation $\mathcal{L}(\H,\J_p)$ for the space of all bounded linear operators from $\H$ to $\J_p$.  Let us apply the projection operator $\mathscr{P}$ to the system (\ref{1}) to obtain
	 \begin{equation}\label{3}
	 \left\{
	 \begin{aligned}
	 \frac{\d\u(t)}{\d t}+\A\u(t)+\B(\u(t))+\mathcal{C}(\u(t))&=\mathscr{P}\f(t), \ t\in(0,T),\\
	 \u(0)=\x, 
	 \end{aligned}\right. 
	 \end{equation}
	 where
	 \begin{align*}
	 \A\u&=-\mathscr{P}\Delta\u,\textrm{ with domain }\D_p(\A)=\D_p(\Delta)\cap\J_p,\\
	 \B(\u)&=\B(\u,\u),\textrm{ with
	 }\B(\u,\v)=\mathscr{P}[(\u\cdot\nabla)\v]=\mathscr{P}[\nabla\cdot(\u\otimes\v)],\\
	 \mathcal{C}(\u)&=\mathscr{P}[|\u|^{r-1}\u],
	 \end{align*}
	 and $\x\in\J_p$. For $r\geq 1$, the operator $\mathcal{C}(\cdot)$ is Gateaux differentiable with the Gateaux derivative 
	 \begin{align}\label{29}
	 \mathcal{C}'(\u)\v&=\left\{\begin{array}{cl}\mathscr{P}(\v),&\text{ for }r=1,\\ \left\{\begin{array}{cc}\mathscr{P}(|\u|\v)+(r-1)\mathscr{P}\left(\frac{\u}{|\u|^{3-r}}(\u\cdot\v)\right),&\ \text{ if }\ \u\neq \mathbf{0},\\\mathbf{0},&\ \text{ if }\ \u=\mathbf{0},\end{array}\right.&\text{ for } 1<r<3,\\ \mathscr{P}(|\u|^{r-1}\v)+(r-1)\mathscr{P}(\u|\u|^{r-3}(\u\cdot\v)), &\text{ for }r\geq 3,\end{array}\right.
	 \end{align}
	 for all $\u,\v\in\L^{p}(\R^d)$, for $p\in[2,\infty)$. It should be recalled that $\mathscr{P}\Delta=\Delta\mathscr{P}$ (cf. \cite{MTSS}), and hence $\A$ is essentially equal to $-\Delta$ and $e^{-t\A}$ is substantially the heat semigroup (Gauss-Weierstrass semigroup, \cite{Kato5}) and is given by $$(e^{-t\A}\u)(x)=\int_{\R^d}\Psi(t,x-y)\u(y)\d y, \ \text{ where }\ \Psi(t,x)=\frac{1}{(4\pi t)^{\frac{d}{2}}}e^{-\frac{|x|^2}{4t}}, \ t>0,\ x\in\R^d,$$ and $\u\in\L^q(\R^d)$, $q\in[1,\infty)$. Thus, the operator system \eqref{3} can be  transformed into a nonlinear integral equation as follows: 
	 \begin{align}\label{4}
	 \u(t)=e^{-t\A}\x-\int_0^te^{-(t-s)\A}[\B(\u(s))+\mathcal{C}(\u(s))]\d s+\int_0^te^{-(t-s)\A}\mathscr{P}\f(s)\d s,
	 \end{align}
	 for all $t\in[0,T]$. For a given $\x\in\J_p$ and $\f\in\mathrm{L}^1(0,T;\J_p)$, a function $\u\in\C([0,T];\J_p),$ for $\max\left\{d,\frac{d(r-1)}{2}\right\}<p<\infty$ satisfying \eqref{4} is called a \emph{mild solution} to the system \eqref{3}.

	 Since $e^{-t\A}$ is an analytic semigroup, we infer that $e^{-t\A}:\L^p\to\L^q$ is a bounded map whenever $1<p\leq q<\infty$ and $t>0$, and there exists a constant $C$ depending on $p$ and $q$ such that (see \cite{Kato5})
	 \begin{align}
	 \|e^{-t\A}\g\|_{q}&\leq Ct^{-\frac{d}{2}\left(\frac{1}{p}-\frac{1}{q}\right)}\|\g\|_p,\label{1.6}\\
	 \|\nabla e^{-t\A}\g\|_{q}&\leq Ct^{-\frac{1}{2}-\frac{d}{2}\left(\frac{1}{p}-\frac{1}{q}\right)}\|\g\|_p,\label{1.5}
	 \end{align}
	 for all $t\in(0,T]$ and $\g\in\L^p(\R^d)$. Using the estimates  \eqref{1.6}-\eqref{1.5}, one can estimate $\|e^{-t\A}\B(\u,\v)\|_p$ as 
	 \begin{align}\label{1.7}
	 \|e^{-t\A}\B(\u,\v)\|_p&\leq Ct^{-\left(\frac{1}{2}+\frac{d}{2p}\right)}\|\u\|_p\|\v\|_p, 
	 \end{align}
	 for all $t\in(0,T]$ and $\u,\v\in\J_p$. Furthermore, using the estimate \eqref{1.6}, we calculate $\|e^{-t\A}\mathcal{C}(\u)\|_{p}$ and $\|e^{-t\A}\mathcal{C}'(\u)\v\|_{p}$ as 
	 \begin{align}
	 \|e^{-t\A}\mathcal{C}(\u)\|_{p}&\leq Ct^{-\frac{d(r-1)}{2p}}\|\u\|_{p}^{r},\label{1.8}\\
	   \|e^{-t\A}\mathcal{C}'(\u)\v\|_{p}&\leq Ct^{-\frac{d(r-1)}{2p}}\|\u\|_{p}^{r-1}\|\v\|_{p},\label{1.9}
	 \end{align}
	  for all $t\in(0,T]$ and $\u,\v\in\J_p$. For the existence of local mild solution in $\L^p$ to the 3D NSE in whole space and bounded domains, the interested readers are referred to see \cite{FJR,FBW1,Kato5,YGTM}, etc. 
	 
	\subsection{Stochastic CBF equations perturbed by L\'evy noise} Let $(\Omega,\mathscr{F},\mathbb{P})$ be a complete probability space equipped with an increasing family of sub-sigma fields $\{\mathscr{F}_t\}_{0\leq t\leq T}$ of $\mathscr{F}$ satisfying  the usual conditions.  On taking the external forcing  as L\'evy noise, one can rewrite the stochastic counterpart of the problem \eqref{3} for $t\in(0,T)$ as 
	  \begin{equation}\label{5}
	 \left\{
	 \begin{aligned}
	 {\d\u(t)}+[\A\u(t)+\B(\u(t))+\mathcal{C}(\u(t))]\d t&=\Phi\d\W(t)+\int_{\mathrm{Z}}\gamma(s-,z)\widetilde{\pi}(\d s,\d z), \\
	 \u(0)&=\x.
	 \end{aligned}\right. 
	 \end{equation}
	 In \eqref{5}, 	 $\mathrm{W}=\{\W(t)\}_{0\leq t\leq T}$ is a cylindrical	 Wiener process and for an orthonormal basis	 $\{e_j(x)\}_{j=1}^{\infty}$ in $\H:=\J_2$, $\mathrm{W}(\cdot)$ can be	 represented as	 $\mathrm{W}(t)=\displaystyle{\sum_{j=1}^{\infty}}e_j(x)\beta_j(t)$,  where $\{\beta_j(\cdot)\}_{j=1}^{\infty}$'s are a sequence of one-dimensional mutually	 independent Brownian motions (\cite{DaZ}). The bounded linear operator
	 $\Phi:\H\to\J_p$, $p\in[2,\infty)$ is a $\gamma$-radonifying
	 operator in $\J_p$  such that
	 \begin{align*}
	 \Phi\d\mathrm{W}(t)=\sum_{j=1}^{\infty}\Phi e_j(x)\d\beta_j(t)=
	 \sum_{j=1}^{\infty}\int_{\R^d}\mathscr{K}(x,y)e_j(y)\d y\d
	 \beta_j(t),
	 \end{align*}
	 where $\mathscr{K}(\cdot,\cdot)$ is the kernel of the operator
	 $\Phi$ (Theorem 2.2, \cite{BLH}). In particular, the operator $\Phi\in\gamma(\H,\J_p)$
	 satisfies 
	 	$$\|\Phi\|_{\gamma(\H,\J_p)}
	 	\equiv\left\{\displaystyle{\int_{\R^d}}\left[\displaystyle{\int_{\R^d}}|\mathscr{K}(x,y)|^2\d
	 	y\right]^{p/2}\d x\right\}^{1/p}<+\infty,$$
	 where $\gamma(\H,\J_p)$ is the space of all $\gamma$-radonifying
	 operators  from $\H$ to $\J_p$.\footnote{Let $\mathbb{U}$ be a real separable Hilbert space and $\mathbb{X}$ be a Banach space. A bounded linear operator $R\in\mathcal{L}(\mathbb{U},\mathbb{X})$ is $\gamma$-radonifying provided that there exists a centered Gaussian probability $\nu$ on $\mathbb{X}$ such that $\int_{\mathbb{X}}\varphi(x)\d\nu(x)=\|R^*\varphi\|_{\mathbb{U}},\ \varphi\in\mathbb{X}^*$. Such a measure is at most one, and hence we set $\|R\|_{\gamma(\mathbb{U},\mathbb{X})}^2:=\int_{\mathbb{X}}\|x\|_{\mathbb{X}}^2\d\nu(x).$ We denote $\gamma(\mathbb{U},\mathbb{X})$ for the space of $\gamma$-radonifying operators, and $\gamma(\mathbb{U},\mathbb{X})$ equipped with the norm $\|\cdot\|_{\gamma(\mathbb{U},\mathbb{X})}$ is a separable Banach space.} 
	 
	 Let us denote by $\Z$, a measurable subspace of some Hilbert space (for example
	 measurable subspaces of $\R^d$, $\L^2(\R^d)$, etc) and $\lambda(\d
	 z)$,  a $\sigma$-finite L\'{e}vy measure on $\Z$ with an associated
	 Poisson random measure $\pi(\d t,\d z)$. We define $\wi\pi(\d t,\d
	 z):=\pi(\d t,\d z)-\lambda(\d z)\d t$ as the compensated
	 Poisson random measure. The jump noise coefficient
	 $\gamma(t,z):=\gamma(t,z,x)$ is such that $\gamma:[0,T]\times \Z\times
	 \J_p \to\J_p$, $p\in[2,\infty)$ and in particular, $\gamma$ satisfies
	 $$\displaystyle{\int_0^T\int_{\Z}}\|\gamma(t,z)\|_p^2\lambda(\d z)\d
	 	t<+\infty.$$ The processes
	 $\mathrm{W}(\cdot)$ and $\pi(\cdot,\cdot)$ are mutually
	 independent. The existence and uniqueness of pathwise strong solutions to the stochastic CBF equations and related models perturbed by Gaussian as well as jump noises in the whole space or periodic domains are available in the literature and the interested readers are referred to see \cite{MRXZ1,ZBGD,ZDRZ,MTM4,MTM6}, etc, and the references therein.

	 We transform the operator system \eqref{3}  into a stochastic nonlinear integral equation as follows: 
	 \begin{equation}\label{6}
	 \begin{aligned}
	 \u(t)&=e^{-t\A}\x-\int_0^te^{-(t-s)\A}[\B(\u(s))+\mathcal{C}(\u(s))]\d s+\int_0^te^{-(t-s)\A}\Phi\d \W(s)\\&\quad+\int_0^t\int_{\Z}e^{-(t-s)\A}\gamma(s-,z)\wi\pi(\d s,\d z),
	 \end{aligned}
 \end{equation}
	 for all $t\in[0,T]$.  The existence of pathwise mild solutions for 2D and 3D NSE perturbed by Gaussian as well as jump noise is  available in \cite{GDJZ,FRS,MTSS,JZZB1}, etc and the references therein. 
	 
	 	\subsection{Stochastic CBF equations perturbed by fractional Brownian noise} Let us now consider the stochastic CBF equations perturbed by fractional Brownian noise as 
	 \begin{equation}\label{1.13}
	\left\{
	\begin{aligned}
	{\d\u(t)}+[\A\u(t)+\B(\u(t))+\mathcal{C}(\u(t))]\d t&=\Phi\d\W^{H}(t), \\
	\u(0)&=\x.
	\end{aligned}\right. 
	\end{equation}
	where $\Phi\in\mathcal{L}(\H,\J_p)$ and $\W^{H}(\cdot)$ is the cylindrical fractional Brownian motion with Hurst parameter $H\in\left(\max\left\{\frac{1}{2},\frac{d}{4}\right\},1\right),$ where $d=2,3$ (see section \ref{sec4} for more details). One can transform the operator system \eqref{1.13}  into a stochastic nonlinear integral equation as
	\begin{equation}\label{1.14}
	\begin{aligned}
	\u(t)&=e^{-t\A}\x-\int_0^te^{-(t-s)\A}[\B(\u(s))+\mathcal{C}(\u(s))]\d s+\int_0^te^{-(t-s)\A}\Phi\d \W^{H}(s),
	\end{aligned}
	\end{equation}
	for all $t\in[0,T]$. For the well-posedness and existence of density for 2D stochastic NSE perturbed by fractional Brownian noise, we refer the interested readers to \cite{LFPS,EHPA}, respectively. 
	 \subsection{Major objectives}
	 The purpose of this work is two folded. 
	 \begin{enumerate}
	 	\item [(i)] In the first part, we show the existence of a unique local mild solution to the deterministic CBF equations \eqref{3} in $\L^p$-spaces, for $\max\left\{d,\frac{d(r-1)}{2}\right\}<p<\infty$, $d\geq 2$ (Theorem \ref{th2.1}). 
	 	\item [(ii)] In the second part, we consider the stochastic counterpart of the problem considered in part (i). 
	 	\begin{itemize}\item [(a)] We first establish the existence and uniqueness of pathwise local mild solutions (up to a random time) to the stochastic CBF equations perturbed by additive L\'evy noise in $\J_p$, for $\max\left\{d,\frac{d(r-1)}{2}\right\}<p<\infty,$ $d\geq 2$ (Theorems \ref{thm3.2} and \ref{thm3.3}). \item [(b)] By considering the noise as fractional Brownian motion, we show the existence and uniqueness of local pathwise mild solutions (up to a random time)  to the stochastic  CBF equations in $\J_p$, for  $\max\left\{d,\frac{d(r-1)}{2}\right\}<p<\infty$ and $\max\left\{\frac{1}{2},\frac{d}{4}\right\}<H<1,$ $d=2,3$ (Theorem \ref{thm3.4}).  \item [(c)]  Finally, we discuss the local solvability of  stochastic CBF equations perturbed by $\alpha$-regular Volterra processes in $\J_p$, for $\max\left\{d,\frac{d(r-1)}{2}\right\}<p<\infty$, $d=2,3$ (Remark \ref{rem4.3}). 
	 		\end{itemize}
	 \end{enumerate}
The difficulty in establishing the existence of local mild solutions to the deterministic CBF equations \eqref{3} lies under estimating the nonlinear terms, which we successfully overcame using the estimates \eqref{1.7}-\eqref{1.9}. On the stochastic counterpart along with these difficulties, additional complication arise due to the presence of noise term (proper regularity of  stochastic convolution). In the case of L\'evy noise,  we handle this obstacle by using the results obtained in \cite{JZZB}. For the fractional Brownian noise and $\alpha$-regular Volterra processes, we overcame this hurdle by using the stochastic convolution results established in \cite{PCBM}. Thus, making use of the estimates \eqref{1.6}-\eqref{1.9} and fixed point arguments (iterative technique) or contraction mapping principle to achieve our goals. It can be easily seen that in the sub-critical and critical cases (that is, for $r\in[1,3]$), the condition on $p$ is $d<p<\infty$, which same as that of NSE (cf. \cite{FJR,FRS}, etc) and for the super-critical case (that is, $r\in(3,\infty)$), the condition on $p$ becomes $\frac{d(r-1)}{2}<p<\infty$.

	\section{Existence and Uniqueness of Deterministic CBF equations}\label{sec2}\setcounter{equation}{0}
In this section, we present the existence and uniqueness of local mild solution to the problem \eqref{3}. We use fixed point arguments (by using a simple iterative technique) to obtain the required result. 

\begin{theorem}\label{th2.1}
For $\max\left\{d,\frac{d(r-1)}{2}\right\}<p<\infty$,	let $\x\in\J_p$ and $\f\in\mathrm{L}^1(0,T;\J_p)$ be given. Then, there exists a time $0<T_*<T$ such that \eqref{3} has a unique mild solution given by \eqref{4} in $\C([0,T_*];\J_p)$. 
\end{theorem}
\begin{proof}
 As discussed in \cite{FJR,Kato5}, etc, in order to prove the theorem, we use an iterative technique. Let us set
 \begin{align}
 \u_0(t)&=e^{-t\A}\x,\\
 \u_{n+1}(t)&=\u_0+\G(\u_n)(t), \ n=0,1,2,\ldots,
 \end{align}
 where $$\G(\u)(t)=-\int_0^te^{-(t-s)\A}[\B(\u(s))+\mathcal{C}(\u(s))]\d s+\int_0^te^{-(t-s)\A}\mathscr{P}\f(s)\d s,$$  which is continuous for all $t\in[0,T]$. Since $e^{-t\A}$ is a contraction semigroup on $\L^p(\R^d)$, first we note that $$\|\u_0(t)\|_{p}=\|e^{-t\A}\x\|_{p}\leq  \|\x\|_{p}.$$ Using the estimates \eqref{1.6}-\eqref{1.9}, we find 
 \begin{align}
 \|\u_{n+1}(t)\|_{p}&\leq\|\x\|_{p}+ \int_0^t\|e^{-(t-s)\A}\B(\u_n(s))\|_{p}\d s+\int_0^t\|e^{-(t-s)\A}\mathcal{C}(\u_n(s))\|_{p}\d s\nonumber\\&\quad+\int_0^t\|e^{-(t-s)\A}\mathscr{P}\f(s)\|_{p}\d s\nonumber\\&\leq \|\x\|_{p}+C\int_0^t(t-s)^{-\left(\frac{1}{2}+\frac{d}{2p}\right)}\|\u_n(s)\|_{p}^2\d s+C\int_0^t(t-s)^{-\frac{d(r-1)}{2p}}\|\u_n(s)\|_{p}^r\d s\nonumber\\&\quad+C\int_0^t\|\f(s)\|_{p}\d s\nonumber\\&\leq \left\{\|\x\|_{p}+C\int_0^t\|\f(s)\|_{p}\d s\right\}+Ct^{\frac{1}{2}-\frac{d}{2p}}\sup_{s\in[0,t]}\|\u_n(s)\|_{p}^2+Ct^{1-\frac{d(r-1)}{2p}}\sup_{s\in[0,t]}\|\u_n(s)\|_{p}^r\nonumber\\&\leq \left\{\|\x\|_{p}+C\int_0^{T}\|\f(s)\|_{p}\d s\right\}+C{T}^{\frac{1}{2}-\frac{d}{2p}}f_n^2+C{T}^{1-\frac{d(r-1)}{2p}}f_n^r, 
 \end{align}
 for all $t\in[0,T]$, where $$f_n=\sup_{t\in[0,T]}\|\u_n(t)\|_{p}.$$ For $f_0=\left\{\|\x\|_{p}+C\int_0^{T}\|\f(s)\|_{p}\d s\right\}$, from the above relation, it is immediate that 
 \begin{align}
 f_{n+1}\leq f_0+C{T}^{\frac{1}{2}-\frac{d}{2p}}f_n^2+C{T}^{1-\frac{d(r-1)}{2p}}f_n^r,\ n=0,1,2,\ldots,
 \end{align} 
 which is a nonlinear recurrence relation. One can easily show by induction that  if $$\frac{1}{2}\min\left\{\left(\frac{1}{4C{T}^{\frac{1}{2}-\frac{d}{2p}}}\right),\left(\frac{1}{4C{T}^{1-\frac{d(r-1)}{2p}}}\right)^{\frac{1}{r-1}}\right\}>f_0,$$ then $$f_n\leq \min\left\{\left(\frac{1}{4C{T}^{\frac{1}{2}-\frac{d}{2p}}}\right),\left(\frac{1}{4C{T}^{1-\frac{d(r-1)}{2p}}}\right)^{\frac{1}{r-1}}\right\}=:K, \ \text{ for all } \ n=1,2,3,\ldots,$$ so that the sequence $\{f_n\}$ is uniformly bounded. 
 
 Let us now consider 
 \begin{align}
\v_{n+2}(t)&= \u_{n+2}(t)-\u_{n+1}(t)\nonumber\\&= -\int_0^te^{-(t-s)\A}[\B(\u_{n+1}(s))-\B(\u_n(s))]\d s-\int_0^te^{-(t-s)\A}[\mathcal{C}(\u_{n+1}(s))-\mathcal{C}(\u_n(s))]\d s,
 \end{align}
 for all $t\in[0,T]$. Once again using the estimates \eqref{1.6}-\eqref{1.9}, we obtain 
 \begin{align}\label{210}
\|\v_{n+2}(t)\|_{p}&=\| \u_{n+2}(t)-\u_{n+1}(t)\|_{p}\nonumber\\& \leq \int_0^t\|e^{-(t-s)\A}\B(\u_{n+1}(s)-\u_{n}(s),\u_{n+1}(s))\|_{p}\d s\nonumber\\&\quad+ \int_0^t\|e^{-(t-s)\A}\B(\u_{n}(s),\u_{n+1}(s)-\u_{n}(s))\|_{p}\d s\nonumber\\&\quad+\int_0^t\left\|e^{-(t-s)\A}\int_0^1\mathcal{C}'(\theta\u_{n+1}(s)+(1-\theta)\u_n(s))(\u_{n+1}(s)-\u_n(s))\d\theta\right\|_{p}\d s\nonumber\\&\leq C\int_0^t(t-s)^{-\left(\frac{1}{2}+\frac{d}{2p}\right)}(\|\u_{n+1}(s)\|_{p}+\|\u_n(s)\|_{p})\|\u_{n+1}(s)-\u_n(s)\|_{p}\d s\nonumber\\&\quad+C\int_0^t(t-s)^{-\frac{d(r-1)}{2p}}(\|\u_{n+1}(s)\|_{p}^{r-1}+\|\u_n(s)\|_{p}^{r-1})\|\u_{n+1}(s)-\u_n(s)\|_{p}\d s\nonumber\\&\leq C{t}^{\frac{1}{2}-\frac{d}{2p}}\sup_{s\in[0,t]}\left(\|\u_{n+1}(s)\|_{p}+\|\u_{n}(s)\|_{p}\right)\sup_{s\in[0,t]}\|\v_{n+1}(s)\|_{p} \nonumber\\&\quad+Ct^{1-\frac{d(r-1)}{2p}}\sup_{s\in[0,t]}\left(\|\u_{n+1}(s)\|_{p}^{r-1}+\|\u_{n}(s)\|_{p}^{r-1}\right)\sup_{s\in[0,t]}\|\v_{n+1}(s)\|_{p},
\end{align}
for all $t\in[0,T]$. Therefore, we deduce that 
\begin{align}\label{211}
\sup_{t\in[0,T]}\| \v_{n+2}(t)\|_{p}&\leq C\left(K{T}^{\frac{1}{2}-\frac{d}{2p}}+K^{r-1}T^{1-\frac{d(r-1)}{2p}}\right)\sup_{t\in[0,T]}\|\v_{n+1}(t)\|_{p}\nonumber\\&\leq C^{n+1}\left(K{T}^{\frac{1}{2}-\frac{d}{2p}}+K^{r-1}T^{1-\frac{d(r-1)}{2p}}\right)^{n+1}\sup_{t\in[0,T]}\|\v_1(t)\|_{p}\nonumber\\&\leq 2KC^{n+1}\left(K{T}^{\frac{1}{2}-\frac{d}{2p}}+K^{r-1}T^{1-\frac{d(r-1)}{2p}}\right)^{n+1}, \ n=0,1,2,\ldots.
\end{align}
Let us now consider the infinite series of the form
\begin{align}\label{212}\u_0(t)+\v_1(t)+\v_2(t)+\cdots+\v_n(t)+\cdots.\end{align} The $n^{\mathrm{th}}$ partial sum of the series is $\u_n(t)$, that is, \begin{align}\label{213}\u_n(t)=\u_0(t)+\sum_{m=0}^{n-1}\v_{m+1}(t).\end{align} Therefore, the sequence $\{\u_n(t)\}$ converges if and only if the  series \eqref{212} converges. From the inequality \eqref{211}, we have 
\begin{align}
&\sup_{t\in[0,T]}\|\u_0(t)\|_{p}+\sum_{m=0}^{\infty}\sup_{t\in[0,T]}\|\v_{m+1}(t)\|_{p}\nonumber\\&\leq \frac{K}{2}+\sum_{m=0}^{\infty}2KC^m\left(K{T}^{\frac{1}{2}-\frac{d}{2p}}+K^{r-1}T^{1-\frac{d(r-1)}{2p}}\right)^m\nonumber\\&=\frac{K}{2}+\frac{2K}{1-C\left(K{T}^{\frac{1}{2}-\frac{d}{2p}}+K^{r-1}T^{1-\frac{d(r-1)}{2p}}\right)}<+\infty,
\end{align}
provided $$C\left(K{T}^{\frac{1}{2}-\frac{d}{2p}}+K^{r-1}T^{1-\frac{d(r-1)}{2p}}\right)<1.$$ Thus, we can choose a time $0<T_*<T$ in such a way that the above condition is satisfied. Therefore the series \eqref{212} converges uniformly in $[0,T_*]$ and we denote the sum of the series by $\u(t)$. Then, the relation \eqref{213} provides $$\u(t)=\lim_{n\to\infty}\u_n(t).$$ The uniform convergence of $\u_n(t)$ to $\u(t)$ and the continuity of the operator $\B(\cdot)+\mathcal{C}(\cdot)$ gives us $$\u(t)=\u_0+\G(\u)(t),$$ which is a mild solution to the problem \eqref{3} in the interval $[0,T_*]$. The continuity of the function $\u(\cdot)$ follows from the uniform convergence and the continuity of the sequence $\{\u_n(\cdot)\}_{n=0}^{\infty}$. 

Let us now show the uniqueness. Let $\u_1(\cdot)$ and $\u_2(\cdot)$ be two local mild solutions of  the problem \eqref{3}. Then $\u=\u_1-\u_2$ satisfies: 
\begin{align}
\u(t)=-\int_0^te^{-(t-s)\A}[\B(\u_1(s))-\B(\u_2(s))]\d s-\int_0^te^{-(t-s)\A}[\mathcal{C}(\u_1(s))-\mathcal{C}(\u_2(s))]\d s.
\end{align}
A calculation similar to \eqref{210} yields 
\begin{align}
\|\u(t)\|_{p}&\leq C{T_*}^{\frac{1}{2}-\frac{d}{2p}}\sup_{t\in[0,T_*]}\left(\|\u_{1}(s)\|_{p}+\|\u_{2}(s)\|_{p}\right)\sup_{t\in[0,T_*]}\|\u(s)\|_{p} \nonumber\\&\quad+C{T_*}^{1-\frac{d(r-1)}{2p}}\sup_{t\in[0,T_*]}\left(\|\u_{1}(s)\|_{p}^{r-1}+\|\u_{2}(s)\|_{p}^{r-1}\right)\sup_{t\in[0,T_*]}\|\u(s)\|_{p}\nonumber\\&\leq C\left(K{T_*}^{\frac{1}{2}-\frac{d}{2p}}+K^{r-1}{T_*}^{1-\frac{d(r-1)}{2p}}\right)\sup_{t\in[0,T_*]}\|\u(s)\|_{p},
\end{align}
for all $t\in[0,T^*]$. One can choose a $T_*$ such that $C\left(K{T_*}^{\frac{1}{2}-\frac{d}{2p}}+K^{r-1}{T_*}^{1-\frac{d(r-1)}{2p}}\right)<1$ and hence the uniqueness of $\u\in\C([0,T_*];\J_p)$ follows. 
\end{proof}

	\section{Existence and Uniqueness of Stochastic CBF equations}\label{sec3}\setcounter{equation}{0}
	This section is devoted for establishing the existence and uniqueness of mild solution up to a random time to the system \eqref{5}. We use the contraction mapping principle to obtain the required result. 
	\subsection{The linear problem}
	For $p\in[2,\infty)$, we know that $e^{-t\A}$ is a $\C_0$-contraction
	semigroup on $\L^p(\R^d)$, and $\L^p(\R^d)$ is an martingale type 2 Banach space and also a $2$-smooth Banach space.
	Let us now consider the \emph{stochastic Stokes equation:}
	\begin{equation}\label{se1}
	\left.
	\begin{aligned}
	\d\w(t)+\A\w(t)\d t&=\Phi\d
	\mathrm{W}(t)+\int_{\Z}\gamma(t-,z)\wi\pi(\d t,\d z),\\
	\w(0)&=\mathbf{0}.
	\end{aligned}
	\right\}
	\end{equation}
Making use of  Theorem 3.6, \cite{JZZB},	the unique solution of the problem (\ref{se1}) with paths in
	$\mathrm{L}^{\infty}(0,T;\J_p)$, $p\in[2,\infty),$ $\mathbb{P}$-a.s., can be
	represented by the stochastic convolution
	\begin{align}\label{se3}
	\w(t)=\int_0^te^{-(t-s)\A}\Phi\d\mathrm{W}(s)+\int_0^t\int_{\Z}e^{-(t-s)\A}\gamma(s-,z)\wi\pi(\d
	s,\d z),
	\end{align}
	for all $t\in[0,T]$, and   (\ref{se3}) has a
	c\`{a}dl\`{a}g modification such that
	\begin{align}\label{se4}
	\mathbb{E}\left[\sup_{0\leq t\leq T}\|\w(t)\|_p^{2}\right]\leq
	C\left(\|\Phi\|^2_{\gamma(\H,\J_p)}T+\int_0^T\int_{\Z}\|\gamma(t,z)\|_p^2\lambda(\d
	z)\d t\right),
	\end{align}
	and $\sup\limits_{0\leq t\leq T}\|\w(t)\|_p<\infty$, $\mathbb{P}$-a.s.
\subsection{The nonlinear problem} Let us now establish the existence of a local mild solution to the stochastic CBF system \eqref{5}. 
\begin{definition}\label{def3.1}
	A $\J_p$-valued  and $\mathscr{F}_{t}$-adapted stochastic process $\u:[0,T]\times \R^d\times \Omega
	\rightarrow \mathbb{R}$ with $\mathbb{P}$-a.s. c\`{a}dl\`{a}g trajectories for $t\in [0,T]$,  is a \emph{mild solution} to the system (\ref{5}), if for any $T>0$, $\u(t):=\u(t,\cdot,\cdot)$ satisfies the integral equation \eqref{6}
	$\mathbb{P}$-a.s., for each $t\in [0,T].$
\end{definition}
Let us set $\v=\u-\w$. Then, $\v (\cdot)$ satisfies the following system $\mathbb{P}$-a.s.:
 \begin{equation}\label{3.4}
\left\{
\begin{aligned}
\frac{\d\v(t)}{\d t}+[\A\v(t)+\B(\v(t)+\w(t))+\mathcal{C}(\v(t)+\w(t))]&=\mathbf{0}, \\
\u(0)&=\x.
\end{aligned}\right. 
\end{equation}
Note that for each fixed $\omega\in\Omega$, \eqref{3.4} is a deterministic system. The operator system \eqref{3.4}  can be transformed  into an nonlinear integral equation as
\begin{align}\label{3.5}
\v(t)=e^{-t\A}\x-\int_0^te^{-(t-s)\A}[\B(\v(s)+\w(s))+\mathcal{C}(\v(s)+\w(s))]\d s,
\end{align}
for all $t\in[0,T]$.   As in the case of deterministic CBF equations, we obtain the existence of a unique local mild solution to the system \eqref{3.4} by using the contraction mapping principle in the space $\C([0,\wi T];\J_p)$, $\mathbb{P}$-a.s., for $\max\left\{d,\frac{d(r-1)}{2}\right\}<p<\infty$, where $0<\wi T<T$ is a random time. Let us set 
\begin{align}
\Sigma(M,\wi T)=\left\{\v\in\C([0,\wi T];\J_p):\|\v(t)\|_{p}\leq M, \ \mathbb{P}\text{-a.s.,} \ \text{ for all }\ t\in[0,\wi T]\right\}. 
\end{align}
Clearly the space $\Sigma(M,\wi T)$ equipped with supremum topology is a complete metric space. 
\begin{theorem}\label{thm3.2}
For $\max\left\{d,\frac{d(r-1)}{2}\right\}<p<\infty$, let the $\mathscr{F}_0$-measurable initial data $\x\in\J_p$, $\mathbb{P}$-a.s. be given. For $M>\|\x\|_{p}$, there exists a random time $\wi T$ such that \eqref{3.4} has a unique mild solution in $\Sigma(M,\wi T)$.  
\end{theorem}
\begin{proof}
	Let us  take any $\v\in \Sigma(M,\wi T)$ and define $\y(t)=\F(\v)(t)$ by 
	\begin{align}
	\y(t)=e^{-t\A}\x-\int_0^te^{-(t-s)\A}[\B(\v(s)+\w(s))+\mathcal{C}(\v(s)+\w(s))]\d s,
	\end{align}
for all $t\in[0,\wi T]$. Let us first establish that $\G:\Sigma(M,\wi T)\to\Sigma(M,\wi T).$ Making use of the estimates \eqref{1.6}-\eqref{1.9}, we find 
	\begin{align}
	\|\y(t)\|_{p}&\leq\|e^{-t\A}\x\|_{p}+\int_0^t\|e^{-(t-s)\A}[\B(\v(s)+\w(s))+\mathcal{C}(\v(s)+\w(s))]\|_{p}\d s\nonumber\\&\leq \|\x\|_{p}+C\int_0^t(t-s)^{-\left(\frac{1}{2}+\frac{d}{2p}\right)}\|\v(s)+\w(s)\|_{p}^2\d s\nonumber\\&\quad+C\int_0^t(t-s)^{-\frac{d(r-1)}{2p}}\|\v(s)+\w(s)\|_{p}^r\d s \nonumber\\&\leq  \|\x\|_{p}+Ct^{\frac{1}{2}-\frac{d}{2p}}\sup_{s\in[0,t]}\|\v(s)+\w(s)\|_{p}^2+Ct^{1-\frac{d(r-1)}{2p}}\sup_{s\in[0,t]}\|\v(s)+\w(s)\|_{p}^r\nonumber\\&\leq \|\x\|_{p}+C{\wi T}^{\frac{1}{2}-\frac{d}{2p}}\left(\sup_{t\in[0,\wi T]}\|\v(t)\|_{p}^2+\sup_{t\in[0,\wi T]}\|\w(t)\|_{p}^2\right)\nonumber\\&\quad+C{\wi T}^{1-\frac{d(r-1)}{2p}}\left(\sup_{t\in[0,\wi T]}\|\v(t)\|_{p}^r+\sup_{t\in[0,\wi T]}\|\w(t)\|_{p}^r\right)\nonumber\\&\leq \|\x\|_{p}+C{\wi T}^{\frac{1}{2}-\frac{d}{2p}}(M^2+\mu_p^2)+C{\wi T}^{1-\frac{d(r-1)}{2p}}(M^r+\mu_p^r),
	\end{align}
$\mathbb{P}$-a.s.,	for all $t\in[0,\wi T]$, where $$\mu_p=\sup_{t\in[0,T]}\|\w(t)\|_{p}.$$ Now, since $M>\|\x\|_{p}$, $\mathbb{P}$-a.s., and $\max\left\{d,\frac{d(r-1)}{2}\right\}<p<\infty$, one can choose $0<\wi T<T$ in such a way that $\|\y(t)\|_{p}\leq M$, for all $t\in[0,\wi T]$, provided 
	\begin{align*}
	\|\x\|_{p}+C{\wi T}^{\frac{1}{2}-\frac{d}{2p}}(M^2+\mu_p^2)+C{\wi T}^{1-\frac{d(r-1)}{2p}}(M^r+\mu_p^r)\leq M. 
	\end{align*}
	Therefore $\y\in\Sigma(M,\wi T)$. 
	
	Our next aim is to show that $\F:\Sigma(M,\wi T)\to\Sigma(M,\wi T)$ is a contraction. Let us consider $\v_1,\v_2\in\Sigma(M,\wi T)$ and set $\y_i(t)=\G(\v_i)(t)$, for all $t\in[0,\wi T]$ and $i\in\{1,2\}$ and $\y=\y_1-\y_2$. Then $\y(\cdot)$ satisfies 
	\begin{align*}
	\y(t)&=-\int_0^te^{-(t-s)\A}[\B(\v_1(s)+\w(s))-\B(\v_2+\w(s))+\mathcal{C}(\v_1(s)+\w(s))-\mathcal{C}(\v_2(s)+\w(s))]\d s,
	\end{align*}
$\mathbb{P}$-a.s.,	for all $t\in[0,\wi T]$. Once again using the bilinearity of $\B(\cdot)$ and Taylor's formula, we find 
	\begin{align}
	\|\y(t)\|_{p}&\leq \int_0^t\|e^{-(t-s)\A}\B(\v_1(s)-\v_2(s),\v_1(s)+\w(s))\|_{p}\d s\nonumber\\&\quad+ \int_0^t\|e^{-(t-s)\A}\B(\v_2(s)+\w(s),\v_1(s)-\v_2(s))\|_{p}\d s\nonumber\\&\quad+\int_0^t\left\|e^{-(t-s)\A}\int_0^1\mathcal{C}'(\theta\v_1(s)+(1-\theta)\v_2(s)+\w(s))(\v_1(s)-\v_2(s))\d\theta\right\|_{p}\d s\nonumber\\&\leq C\int_0^t(t-s)^{-\left(\frac{1}{2}+\frac{d}{2p}\right)}\|\v_1(s)+\w(s)\|_{p}\|\v_1(s)-\v_2(s)\|_{p}\d s\nonumber\\&\quad+C\int_0^t(t-s)^{-\left(\frac{1}{2}+\frac{d}{2p}\right)}\|\v_2(s)+\w(s)\|_{p}\|\v_1(s)-\v_2(s)\|_{p}\d s\nonumber\\&\quad+C\int_0^t(t-s)^{-\frac{d(r-1)}{2p}}\left(\|\v_1(s)\|_{p}+\|\v_2(s)\|_{p}+\|\w(s)\|_{p}\right)^{r-1}\|\v_1(s)-\v_2(s)\|_{p}\d s \nonumber\\&\leq C{t}^{\frac{1}{2}-\frac{d}{2p}}\sup_{s\in[0,t]}\left(\|\v_1(s)\|_{p}+\|\v_2(s)\|_{p}+\|\w(s)\|_{p}\right)\sup_{s\in[0,t]}\|\y(s)\|_{p}\nonumber\\&\quad+Ct^{1-\frac{d(r-1)}{2p}}\sup_{s\in[0,t]}\left(\|\v_1(s)\|_{p}^{r-1}+\|\v_2(s)\|_{p}^{r-1}+\|\w(s)\|_{p}^{r-1}\right)\sup_{s\in[0,t]}\|\y(s)\|_{p}\nonumber\\&\leq C\left({\wi T}^{\frac{1}{2}-\frac{d}{2p}}(M+\mu_p)+{\wi T}^{1-\frac{d(r-1)}{2p}}(M^{r-1}+\mu_p^{r-1})\right)\sup_{t\in[0,\wi T]}\|\y(t)\|_{p},
	\end{align}
	for all $t\in[0,\wi T]$. For $\max\left\{d,\frac{d(r-1)}{2}\right\}<p<\infty$, one can choose $0<\wi T<T$ in such a way that $$C\left({\wi T}^{\frac{1}{2}-\frac{d}{2p}}(M+\mu_p)+{\wi T}^{1-\frac{d(r-1)}{2p}}(M^{r-1}+\mu_p^{r-1})\right)<1.$$ Hence, $\F$ is a strict contraction in $\Sigma(M,\wi T)$ and an application of the contraction mapping principle provides the existence of mild solution to the problem \eqref{3.4} up to a random time $0<\wi T<T$. Uniqueness follows form the representation \eqref{3.5}. 
\end{proof}
Since $\u=\v+\w$, we immediately obtain the following Theorem on the existence of mild solution to the system  \eqref{5}.
\begin{theorem}\label{thm3.3}
	For $\max\left\{d,\frac{d(r-1)}{2}\right\}<p<\infty$, let the $\mathscr{F}_0$-measurable initial data $\x\in\J_p$, $\mathbb{P}$-a.s. be given. Then there exists a random time $0<\wi T<T$ such that \eqref{5} has a unique mild solution $\u\in\mathrm{L}^{\infty}(0,\wi T;\J_p)$, $\mathbb{P}$-a.s. with a c\`{a}dl\`{a}g modification.
\end{theorem}

	\section{Stochastic CBF equations subjected to fraction Brownian motion}\label{sec4}\setcounter{equation}{0} In this section, we obtain the existence and uniqueness of a local mild solution up to a random time for the stochastic CBF equations \eqref{1.13}, for $d=2,3$.  
	\subsection{Fractional Brownian motion}
	The first study on fractional Brownian motion (fBm) within the Hilbertian framework is reported in \cite{ANK}. Due to various practical applications, the stochastic analysis of fBm has been intensively developed starting from the nineties.  For a comprehensive study, the interested readers are referred to see \cite{VPMT,DNu}, etc. In this subsection, we provide a brief description of fBm and its stochastic integral representation in separable Hilbert spaces (cf. sections 4 and 5, \cite{EIMR} for separable Banach spaces). Let us consider a time interval $[0,T]$, where $T$ is an arbitrary fixed time horizon. 
	\begin{definition}
	A fractional Brownian motion (fBm) with Hurst parameter $H\in(0,1)$ is a centered Gaussian process $\W^H$ with covariance $$R_{H}(t,s):= \E\left[\W^{H}(t)\W^{H}(s)\right]=\frac{1}{2}\left(t^{2H}+s^{2H}-|t-s|^{2H}\right),$$ where $s,t\in[0,T]$. 
		\end{definition}Note that if $H=\frac{1}{2}$, then $\W^{\frac{1}{2}}$ is the standard Brownian motion. It should be recalled that fBm is not a Markov process except in the case $H=\frac{1}{2}$. The fBm is the only $H$-self-similar Gaussian process (that is, for
	any constant $a > 0$, the processes $\{a^{-H}\W^{H}(at)\}_{0\leq t\leq T}$ and $\W^{H}=\{\W^{H}(t)\}_{0\leq t\leq T}$ have
the same distribution) with stationary increments (Proposition 1.1, \cite{CAT}) $$\E\left[(\W^{H}(t)-\W^{H}(s))^2\right]=|t-s|^{2H}.$$ Furthermore, the process $\W^{H}$ admits the  Wiener integral representation of the form
	\begin{align}\label{41}\W^{H}(t)=\int_0^tK_H(t,s)\d\W(s),\end{align} where $\W =\{\W(t)\}_{0\leq t\leq T}$ is a Wiener process, and $K_H(\cdot,\cdot)$ is the kernel given by $$K_{H}(t,s)=d_{H}(t-s)^{H-\frac{1}{2}}+s^{H-\frac{1}{2}}\F\left(\frac{t}{s}\right),$$ where $d_H$ is a constant  and $$\F(z)=d_H\left(\frac{1}{2}-H\right)\int_0^{z-1}\theta^{H-\frac{3}{2}}\left(1-(\theta+1)^{H-\frac{1}{2}}\right)\d\theta.$$ For $H>\frac{1}{2}$, the kernel $K_H(\cdot,\cdot)$ has the simpler expression $$K_H(t,s)=c_Hs^{\frac{1}{2}-H}\int_s^t(u-s)^{H-\frac{3}{2}}u^{H-\frac{1}{2}}\d u,$$ where $t>s$ and $c_H=\left(\frac{H(H-1)}{\beta(2-2H,H-\frac{1}{2})}\right)^{\frac{1}{2}},$ $\beta(\cdot,\cdot)$ being the beta function. The fact that the process defined by \eqref{41} is a fBm follows from the equality $$\int_0^{t\wedge s}K_H(t,u)K_H(s,u)\d u=R_H(t,s).$$ Moreover, the kernel $K_H(\cdot,\cdot)$ satisfies the condition $$\frac{\partial}{\partial t}K_H(t,s)=d_H\left(H-\frac{1}{2}\right)\left(\frac{s}{t}\right)^{\frac{1}{2}-H}(t-s)^{H-\frac{3}{2}}.$$ Note that  the fBm is an $\alpha$-regular Volterra process  for $\alpha=H-\frac{1}{2}$, where $H>\frac{1}{2}$ (see Remark 2.2, \cite{PCBM} for more details).  
	
	 Let $\mathbb{U}$ be  a  separable  Hilbert  space  with  scalar  product $(\cdot,\cdot)$.  Let $\mathcal{E}_{H}$ denote the linear space of $\mathbb{U}$-valued step functions on $[0,T]$ of the form \begin{align}\label{42}\varphi(t)=\sum_{i=0}^{m-1}x_i\mathds{1}_{[t_{i},t_{i+1})}(t),\end{align} where $0=t_0,t_1,t_2,\ldots,t_m\in[0,T]$, $m\in\mathbb{N}$, $x_i\in\mathbb{U}$. The space $\mathcal{E}_{H}$ is equipped with the inner product $$\left(\sum_{i=0}^{m-1}x_i\mathds{1}_{[0,t_i)},\sum_{j=0}^{n-1}y_j\mathds{1}_{[0,s_j)}\right)_{\mathcal{H}}=\sum_{i=0}^{m-1}\sum_{j=0}^{n-1}(x_i,y_j)R_{H}(t_i,s_j).$$ Note that $\mathcal{E}_H$ is a pre-Hilbert space and we denote the completion of $\mathcal{E}_H$  with respect to $(\cdot,\cdot)_{\mathcal{H}}$ by $\mathcal{H}$.  For $\varphi\in\mathcal{E}_{H}$ of the form \eqref{42}, let us define its Wiener integral with respect to the fBm as $$\int_0^T\varphi(s)\d\W^{H}(s)=\sum_{i=0}^{m-1}x_i(\W^{H}(t_{i+1})-\W^{H}(t_i)).$$ It is clear that the mapping $\varphi=\sum_{i=1}^mx_i\mathds{1}_{(t_{i},t_{i+1}]}\mapsto\int_0^T\varphi(s)\d\W^{H}(s)$ is an isometry between $\mathcal{E}_{H}$ and the linear space $\mathrm{span}\{\W^{H}(t):t\in[0,T]\}$ viewed as a subspace of $\mathrm{L}^2(\Omega;\mathbb{U})$, since $$\E\left[\left\|\int_0^T\varphi(s)\d\W^{H}(s)\right\|_{\mathbb{U}}^2\right]=\|\varphi\|_{\mathcal{H}}^2.$$  The image of an element $\varphi\in\mathcal{H}$ under this isometry is called the Wiener integral of $\varphi$ with respect to the fBm $\W^H$. For $0<s<T$, we consider the operator $K^*:\mathcal{E}_H\to\mathrm{L}^2(0,T;\mathbb{U})$ as $$(K_T^*\varphi)(s)=K(T,s)\varphi(s)+\int_s^T(\varphi(r)-\varphi(s))\frac{\partial K}{\partial r}(r,s)\d r.$$ For $H>\frac{1}{2}$, the operator $K^*$ has the simpler expression 
	$$(K_T^*\varphi)(s)=\int_s^T\varphi(r)\frac{\partial K}{\partial r}(r,s)\d r.$$ The integrals appearing on the right-hand side are both Bochner integrals. Since the operator $K^*$ satisfies $(K^*\varphi,K^*\psi)_{\mathrm{L}^2(0,T;\mathbb{U})}=(\varphi,\psi)_{\mathcal{H}},\ \text{ for all }\ \varphi,\psi\in\mathcal{E}_H,$  $K^*$ can be extended to an isometry between $\mathcal{H}$ and $\mathrm{L}^2(0,T;\mathbb{U})$ in the sense that $$\E\left[\left\|\int_0^T\varphi(s)\d\W^{H}(s)\right\|_{\mathbb{U}}^2\right]=\|K^*\varphi\|_{\mathrm{L}^2(0,T;\mathbb{U})}^2=\|\varphi\|_{\mathcal{H}}^2,\ \text{ for all }\ \varphi\in\mathcal{H}.$$  Hence we have the following connection with the Wiener process $\W$  \begin{align}\label{43}\int_0^t\varphi(s)\d\W^{H}(s)=\int_0^t(K_t^*\varphi)(s)\d\W( s),\end{align} for every $t\in[0,T]$, and $\varphi\mathds{1}_{[0,t]}\in\mathcal{H}$ if and only if $K^*\varphi\in\mathrm{L}^2(0,T;\mathbb{U})$. Furthermore, if $\varphi,\psi\in\mathcal{H}$ are such that $\int_0^T\int_0^T|\varphi(t)||\psi(t)||t-s|^{2H-2}\d s\d t<\infty,$ then their scalar product in $\mathcal{H}$ is given by \begin{align*}(\varphi,\psi)_{\mathcal{H}}=\int_0^T\int_0^T\varphi(t)\psi(t)|t-s|^{2H-2}\d s\d t.\end{align*}
	In general, careful justification is needed for   the existence of right hand side of \eqref{43} (cf. section 5.1, \cite{DNu}). As we are discussing the case of Wiener integrals over the Hilbert space $\mathbb{U}$, we point out that if  $\varphi\in\mathrm{L}^2(0,T;\mathbb{U})$ is a deterministic function, then the relation \eqref{43} holds, and the right hand is well defined in $\mathrm{L}^2(\Omega;\mathbb{U})$ if $K_H^*\varphi\in\mathrm{L}^2(0,T;\mathbb{U})$.
	\subsection{Cylindrical Brownian motion}  For a Hilbert space $\mathbb{U}$, let us now define  the standard cylindrical fractional Brownian motion in $\mathbb{U}$  as the formal series (cf. \cite{TED,STCAT})
	\begin{align}\label{44}
	\W^{H}(t)=\sum_{n=0}^{\infty}e_n\W_n^{H}(t), 
	\end{align}
	where $\{e_n\}_{n=1}^{\infty}$ is a complete orthonormal basis in $\mathbb{U}$ and $\W_n^{H}$ is an one dimensional fBm. It should be noted that the series \eqref{44} does not converge in $\mathrm{L}^2(\Omega;\mathbb{U}),$  and thus $\W^{H}(t)$ is not a well-defined $\mathbb{U}$-valued random variable. But, one can consider a Hilbert space $\mathbb{U}_1$ such that $\mathbb{U}\subset\mathbb{U}_1,$ the linear embedding is a Hilbert-Schmidt operator, therefore, the series \eqref{44} defines a $\mathbb{U}_1$-valued Gaussian random variable and $\{\W^{H}(t)\}_{t\in[0,T]}$ is a $\mathbb{U}_1$-valued cylindrical fBm.
	
Let $\mathbb{Y}$ be an another real and separable Hilbert space and $\mathcal{L}_2(\mathbb{U},\mathbb{Y})$ denote the space of Hilbert-Schmidt operators from $\mathbb{U}$ to $\mathbb{Y}$.  As discussed in \cite{DaZ}, it is possible to define a stochastic integral of the form: 
	\begin{align}\label{45}
	\int_0^T\varphi(t)\d\W^{H}(t),
	\end{align}
	where $\varphi:[0,T]\mapsto\mathcal{L}(\mathbb{U},\mathbb{Y})$, and the integral \eqref{45} is a $\mathbb{Y}$-valued random variable, which is independent of choice of $\mathbb{U}_1$. Let $\varphi$ be a deterministic function with values in $\mathcal{L}_2(\mathbb{U},\mathbb{Y})$ satisfying: 
	\begin{enumerate}
		\item [(i)] for each $x\in\mathbb{U}$, $\varphi(\cdot)x\in\mathrm{L}^p(0,T;\mathbb{Y})$, for $p>\frac{1}{H}$,
		\item [(ii)] $\int_0^T\int_0^T\|\varphi(s)\|_{\mathcal{L}_2(\mathbb{U},\mathbb{Y})}\|\varphi(t)\|_{\mathcal{L}_2(\mathbb{U},\mathbb{Y})}|s-t|^{2H-2}\d s\d t<\infty.$
	\end{enumerate}
Then the stochastic integral \eqref{45} can be expressed as 
\begin{align}\label{46}
	\int_0^T\varphi(t)\d\W^{H}(t):=\sum_{n=1}^{\infty}\int_0^t\varphi(s)e_n\d\W_n^{H}(s)=\sum_{n=1}^{\infty}\int_0^t(K_H^*\varphi e_n)\d\W_n(s),
\end{align}
where $\W_n$ is the standard Brownian motion connected to fBm $\W_n^{H}$ by the representation formula \eqref{41}. Since $H\in(\frac{1}{2},1)$ implies $\varphi e_n\in\mathrm{L}^2(0,T;\mathbb{Y})$, for each $n\in\mathbb{N}$,  so that the terms of the series \eqref{46} are well-defined. Moreover, the sequence of random variables $\left\{\int_0^t\varphi(s)e_n\d\W_n^{H}(s)\right\}_{n=1}^{\infty}$ are mutually independent Gaussian random variables (cf. \cite{TED}).

	For cylindrical Brownian motions  in a separable Banach space $\mathbb{Y}$, the interested readers are referred to see sections 4 and 5, \cite{EIMR}. For stochastic integrals in $\mathbb{Y}$, a series expansion  similar to \eqref{46} is available, where the Hilbert-Schmidt operators from $\mathbb{U}$ to $\mathbb{Y}$  are replaced by $\gamma$-radonifying operators from  $\mathbb{U}$ to $\mathbb{Y}$ (see \cite{EIMR} for more details). One can refer \cite{JCOC,MMEM}, etc for the local solvability in $\mathbb{L}^p$-spaces for  some mathematical models like semilinear heat equation, Hardy-H\'enon parabolic equations, etc perturbed by fBm. 
	
	\subsection{SCBF equations perturbed by fractional Brownian motion} 	
	We consider $\mathbb{U}=\mathbb{H}=\J_2$, $\{e_j\}_{j=1}^{\infty}$ as the complete orthonormal basis of $\J_2$, and  we take $d=2,3$. Next, we consider the following  \emph{stochastic Stokes equation} perturbed by fractional Brownian noise as 
	\begin{equation}\label{47}
	\left.
	\begin{aligned}
	\d\w(t)+\A\w(t)\d t&=\Phi\d\W^{H}(t),\\
	\w(0)&=\mathbf{0},
	\end{aligned}
	\right\}
	\end{equation}
where $\Phi\in\mathcal{L}(\H,\J_p)$ and $\W^{H}=\{\W^{H}(t)\}_{t\in[0,T]}$ is a cylindrical fractional Brownian process. Since the operator $\A$ generates an analytic semigroup on $\J_p$, by standard estimates on Green's function, we obtain (cf. \cite{PCBM}) \begin{align}\label{4.7}\|\S(t)\Phi\|_{\gamma(\H,\J_p)}\leq Ct^{-\frac{d}{4}},\ \text{ for }\ t>0.\end{align} 	Using Corollary 4.1, \cite{PCBM} (see Remark 4.2 and Section 5.2 also), under the assumption $$\max\left\{\frac{1}{2},\frac{d}{4}\right\}<H<1,$$	the unique solution of the problem (\ref{se1}) with paths in
	$\mathrm{C}([0,T];\J_p)$, $p\in[2,\infty),$ $\mathbb{P}$-a.s., can be
	represented by the stochastic convolution
	\begin{align}\label{48}
	\w(t)=\int_0^te^{-(t-s)\A}\Phi\d\mathrm{W}^{H}(s),
	\end{align}
	for all $t\in[0,T]$ has a modification such that
	\begin{align}\label{49}
\sup_{t\in[0,T]}\left\|\int_0^te^{-(t-s)\A}\d\mathrm{W}^{H}(s)\right\|_{p}<\infty, \ \mathbb{P}\text{-a.s.}
	\end{align}
	Then the following theorem can be established in a similar way as that of Theorems \ref{thm3.2} and \ref{thm3.3}. 
	
	\begin{theorem}\label{thm3.4}
		For $\max\left\{d,\frac{d(r-1)}{2}\right\}<p<\infty$ and $\max\left\{\frac{1}{2},\frac{d}{4}\right\}<H<1,$ $d=2,3,$ let the $\mathscr{F}_0$-measurable initial data $\x\in\J_p$, $\mathbb{P}$-a.s. be given. Then there exists a random time $\widehat{T}$ such that \eqref{5} has a unique mild solution $\u\in\mathrm{C}([0,\widehat{T}];\J_p)$, $\mathbb{P}$-a.s. satisfying \eqref{1.14}. 
	\end{theorem}

\begin{remark}\label{rem4.3}
	One can also consider the stochastic CBF equations perturbed by $\alpha$-regular Volterra processes as 
	  \begin{equation}\label{50}
	 \left\{
	 \begin{aligned}
	 {\d\u(t)}+[\A\u(t)+\B(\u(t))+\mathcal{C}(\u(t))]\d t&=\Phi\d\mathcal{B}(t), \\
	 \u(0)&=\x,
	 \end{aligned}\right. 
	 \end{equation}
	 where $\Phi\in\mathcal{L}(\H,\J_p)$ satisfies 
	 \eqref{4.7} and $\mathcal{B}$ is an infinite-dimensional $\alpha$-regular cylindrical Volterra process with $\alpha\in(0,\frac{1}{2})$, which belongs to a finite Wiener chaos (see \cite{PCBM} for more details on $\alpha$-regular Volterra processes). Then for $$\alpha>\frac{d}{4}-\frac{1}{2},\ d=2,3,$$ the process $$	\w(t)=\int_0^te^{-(t-s)\A}\Phi\d\mathcal{B}(s),$$ has a modification in $\mathrm{C}([0,T];\J_p)$, $p\in[\frac{2}{1+2\alpha},\infty),$ $\mathbb{P}$-a.s. Thus a result similar to Theorem \ref{thm3.4} can be obtained in this case also for the system \eqref{50}, that is, the existence and uniqueness of a  mild solution $$\u(t)=e^{-t\A}\x-\int_0^te^{-(t-s)\A}[\B(\u(s))+\mathcal{C}(\u(s))]\d s+\int_0^te^{-(t-s)\A}\Phi\d\mathcal{B}(s),$$ for $t\in[0,\overline{T}]$, where $0<\overline{T}<T$ is a random time,  to the system \eqref{50} with $\mathbb{P}$-a.s. continuous modification with trajectories in $\J_p$, for $\max\left\{d,\frac{d(r-1)}{2}\right\}<p<\infty$. 
\end{remark}

\vskip 0.2 cm 
\noindent\textbf{Conclusions and future plans:} The existence and uniqueness of a local mild solution in $\L^p(\R^d)$ with $\max\left\{d,\frac{d(r-1)}{2}\right\}<p<\infty$ for deterministic and stochastic CBF equations in $\R^d$ (for various kinds of noises)  is established in this work. The case of $p=\max\left\{d,\frac{d(r-1)}{2}\right\}$  is an interesting problem and it will be addressed in a future work (for similar works, see \cite{Kato5} for the deterministic NSE and \cite{MTSS} for stochastic NSE).

\vskip 0.2 cm 
 \medskip\noindent
{\bf Acknowledgments:} M. T. Mohan would  like to thank the Department of Science and Technology (DST), India for Innovation in Science Pursuit for Inspired Research (INSPIRE) Faculty Award (IFA17-MA110). 

\vskip 0.2 cm 
\noindent\textbf{Conflict of interest:} The author has no conflicts of interest to declare that are relevant to the content of this article.


\begin{thebibliography}{99}
	


	
	
	
	
	
	
	
	
	
	
	
	
	
	
	
	

	
	
	
	
	



\bibitem{BLH} 
\newblock Z. Brze\'{z}niak and H. Long,
\newblock {A note on $\gamma$-radonifying and summing operators},
\newblock \emph{Stochastic Analysis}, Banach Center Publications, Institute of Mathematics, Polish Academy of Sciences, Warszawa, \textbf{105} (2015), 43--57.


\bibitem{ZBGD}  Z. Brze\'zniak and Gaurav Dhariwal, Stochastic tamed Navier-Stokes equations on $\mathbb{R}^3$: the existence and the uniqueness of solutions and the existence of an invariant measure,  \emph{Journal of Mathematical Fluid Mechanics}, {\bf 22}, Article number: 23 (2020).


\bibitem{BHZ} 
\newblock Z. Brze\'{z}niak, E. Hausenblas and J. Zhu,
\newblock {Maximal inequality for stochastic convolutions driven by compensated Poisson random measures in Banach spaces},
\newblock \emph{Ann. Inst. Henri Poincar\'e Probab. Stat.,} \textbf{53} (2017), 937--956.

	
	

	
	
	\bibitem{ZCQJ} 	Z. Cai and Q. Jiu, Weak and Strong solutions for the incompressible Navier-Stokes equations with damping, \emph{Journal of Mathematical Analysis and Applications}, {\bf 343} (2008), 799--809.
	
	
	
	
	
	
		
		
	
	
	
	
	
	
	
	
	
	
	

	
\bibitem{JCOC}	J. Clarke and C. Olivera, Local $L^p$-solution for semilinear heat equation with fractional noise, \url{https://arxiv.org/abs/1902.06084}.
	
	\bibitem{PCBM}	P. Coupek,  B. Maslowski and M. Ondrejat, $L^p$-valued stochastic convolution integral driven by Volterra noise, \emph{Stoch. Dyn.}, {\bf 18}(6) (2018), 1850048.
	
		\bibitem{DaZ}
	\newblock G. Da Prato and J. Zabczyk,
	\newblock \emph{Stochastic Equations in Infinite Dimensions},
	\newblock Cambridge University Press, 1992.
	
		\bibitem{GDJZ}
	\newblock G. Da Prato and J. Zabczyk,
	\newblock \emph{Ergodicity for Infinite Dimensional Systems},
	\newblock London Mathematical Society Lecture Notes, {\bf 229}, Cambridge
	University Press, 1996.
	
		
		
	
	
	
\bibitem{ZDRZ} Z.  Dong and R.  Zhang,	3D tamed Navier-Stokes equations driven by multiplicative L\'evy noise: Existence, uniqueness and large deviations, \url{https://arxiv.org/pdf/1810.08868.pdf}.


\bibitem{TED}  T.E. Duncan, B. Pasik-Duncan,and B. Maslowski, Fractional Brownian motion and stochastic equations in Hilbert spaces, \emph{Stoch. Dyn.} {\bf 2} (2002), 225--250. 
	
	
	
	
	
	
	
	
	
	
	

 \bibitem{FJR} 
\newblock E. B. Fabes, B. F. Jones and N. M. Riviere,
\newblock {The Initial value problem for the Navier-Stokes equations with data in $\L^p$},
\newblock \emph{Archive for Rational Mechanics and Analysis}, \textbf{45} (1972), 222--240.


 \bibitem{LFPS} L. Fang, and  P. Sundar and F. G. Viens,  Two-dimensional stochastic Navier-Stokes equations with fractional Brownian noise, \emph{Random Oper. Stoch. Equ.}, {\bf 21}(2) (2013),  135--158.

 

	
	
	
	\bibitem {FRS} 
	\newblock B. P. W. Fernando, B. R\"{u}diger and S. S. Sritharan,
	\newblock {Mild solutions of stochastic Navier-Stokes equation with jump noise in $\L^p$-spaces},
	\newblock \emph{Mathematische Nachrichten}, \textbf{288} (2015), 1615--1621.

	\bibitem{KWH}	K. W. Hajduk and J. C. Robinson, Energy equality for the 3D critical convective Brinkman-Forchheimer equations, \emph{Journal of Differential Equations}, {\bf 263} (2017), 7141--7161.
	
	
	

\bibitem{EHPA}  E. Hausenblas and P. A. Razafimandimby,  Existence of a density of the 2-dimensional stochastic Navier Stokes equation driven by L\'evy processes or fractional Brownian motion, \emph{Stochastic Process. Appl.}, {\bf  130}(7) (2020), 4174--4205.
	

	
	
	
	
	
	
	
	
	
	
	\bibitem{Kato5} 
	\newblock T. Kato,
	\newblock {Strong $\L^p$-solutions of the Navier-Stokes equation in $\R^m$, with applications to weak solutions},
	\newblock \emph{Mathematische Zeitschrift}, \textbf{187} (1984), 471--480.
	

	
	
	



\bibitem{YGTM} Y. Giga and T. Miyakawa, Solutions in $\L^r$ of the Navier-Stokes initial value problem, \emph{Arch. Ration. Mech. Anal.}, {\bf 89}(3) (1985), 267--281.

\bibitem{EIMR}  E. Issoglio,  and M.  Riedle, 
Cylindrical fractional Brownian motion in Banach spaces,
\emph{Stochastic Process. Appl.}, {\bf 124 } (11) (2014), 3507--3534.

\bibitem{ANK} A.N. Kolmogorov,  Wienerische Spiralen und einige andere interessante Kurven im Hilbertschen Raum, \emph{C. R. (Doklady). Acad. URSS (N.S.)}, {\bf 26} (1940), 115--118.

	
	
		
	
	



\bibitem{MMEM} M. Majdoub and E. Mliki, Well-posedness for Hardy-H\'enon parabolic equations with fractional Brownian noise, \emph{Analysis and Mathematical Physics} {11:20} (2021). 

	\bibitem{PAM}	P.A. Markowich, E.S. Titi and S. Trabelsi,	Continuous data assimilation for the three-dimensional Brinkman-Forchheimer-extended Darcy model, \emph{Nonlinearity}, {\bf 29}(4), 2016, 1292-1328.


	
		
	
	



	
	
		
	
	
	
	
\bibitem{MTSS} 	M. T. Mohan and S. S.  Sritharan, 	$\mathbb{L}^p$-solutions of the stochastic Navier-Stokes equations subject to L\'evy noise with $\mathbb{L}^m(\R^m)$-initial data, \emph{	Evol. Equ. Control Theory}, {\bf 6}(3) (2017), 409--425.
	
	

	
	
	
	
	
	
	
	
	
	\bibitem{MTM5} M. T. Mohan, 	On convective Brinkman-Forchheimer equations, \emph{Submitted}.  
	
	\bibitem{MTM4} M. T. Mohan, 	Stochastic convective Brinkman-Forchheimer equations, \emph{Submitted}, \url{https://arxiv.org/abs/2007.09376}.  
	
	\bibitem{MTM6} M. T. Mohan, Well-posedness and asymptotic behavior of the stochastic convective Brinkman-Forchheimer equations perturbed by pure jump noise, \emph{Submitted}, \url{https://arxiv.org/abs/2008.08577}.  
	
	
	
	
	
	


\bibitem{DNu}  D. Nualart, \emph{The Malliavin calculus and related topics}, 2nd Ed. Probability and Its Application (New York), Springer, Berlin (2006).

\bibitem{VPMT}	V.Pipiras and M.Taqqu, Integration questions related to the fractional Brownian motion, \emph{Probab.
	Theory Relat. Fields}, {\bf 118}(2)  (2001), 251--281.
	
	
	
	
		
		
		
		
		
			\bibitem{MRXZ1}	M. R\"ockner and X. Zhang, Stochastic tamed 3D Navier-Stokes equation: existence, uniqueness and ergodicity, \emph{Probability Theory and Related Fields}, {\bf 145} (2009) 211--267.
			
		
			
	
	
	
	
	
	

	
	
	
	

\bibitem{STCAT} S. Tindel, C. A. Tudor,  and F.  Viens,  Stochastic evolution equations with fractional Brownian motion, \emph{Probab. Theory Related Fields}, {\bf 127}(2) (2003), 186--204.

\bibitem{CAT} C. A. Tudor, \emph{Analysis of Variations for Self-similar Processes, A Stochastic Calculus Approach}, Springer International Publishing Switzerland 2013.
	


\bibitem{FBW1}  F. B. Weissler, The Navier-Stokes initial value problem in $\mathbb{L}^p$, \emph{Arch. Ration. Mech. Anal.}. {\bf  74}, 219--230 (1980).
	

	
	
	
	
	
\bibitem{JZZB} 	J. Zhu, Z. Brze\'zniak,   and W. Liu, \emph{Maximal inequalities and exponential estimates for stochastic convolutions driven by L\'evy-type processes in Banach spaces with application to stochastic quasi-geostrophic equations},
\emph{SIAM J. Math. Anal.}, {\bf 51} (3) (2019), pp. 2121--2167. 


\bibitem{JZZB1} J. Zhu,  Z. Brze\'zniak, and W. Liu, 
$\L^p$-solutions for stochastic Navier-Stokes equations with jump noise,
\emph{Statist. Probab. Lett.}, {\bf 155} (2019), 108563, 9 pp.
	
	
	
	

	
\end{thebibliography}
\end{document}